\documentclass[12pt]{amsart}
\usepackage{amsmath}
\usepackage{amsfonts}
\usepackage{amsthm}
\usepackage{amssymb}
\usepackage{amscd}
\usepackage{epic}
\usepackage{eepic}
\usepackage{bigints}
\newcommand{\printname}[1]
  {\smash{\makebox[0pt]{pace{-2.0in}\raisebox{8pt}{\tiny #1}}}}
\newtheorem {Theorem}   {Theorem}
\newtheorem {Lemma} {Lemma} [section]
\newtheorem {Proposition}[Lemma]{Proposition}

\theoremstyle{definition}

\theoremstyle{remark}

\newtheorem {Corollary}[Lemma]{Corollary}
\newcommand {\cal}[1]  {{\mathcal{#1}}}
\newcommand{\ssminus}{{\smallsetminus}}
\newcommand{\nab}[1][]{\ensuremath{\mathrm{\nabla}{#1}}}
\newcommand{\be}{\begin{equation}}
\newcommand{\ee}{\end{equation}}
\newcommand{\lb}[1]{\label{#1}}

\newcommand{\bet}{\beta}
\newcommand{\al}{\alpha}
\newcommand{\om}{\omega}

\newcommand{\we}{\wedge}

\newcommand{\Ref}[1]{(\ref{#1})}
\newcommand{\fr}{\frac}

\newcommand{\ra}{\rightarrow}

\newcommand{\ri}{\mathrm{r}}

\newcommand{\Lap}{\Delta}

\newcommand{\wht}{\widehat}
\newcommand{\sk}{SKR}
\def\t{\tau}

\newcommand{\sols}{Ricci solitons}

\newcommand{\cc}{\lambda}

\begin{document}
%\begin{article}
%\begin{opening}
\title{Almost soliton duality}
\author{Gideon Maschler}
%\quad Draft \today}
% \author{}

%\date{\today}

\address{Department of Mathematics and Computer Science, Clark University,
Worcester, Massachusetts 01610, U.S.A.}
\email{gmaschler@clarku.edu}

%\thanks{}

%\maketitle

\begin{abstract}
Gradient Ricci almost solitons were introduced by Pigola, Rigoli, Rimoldi and Setti \cite{PRRS}.
They are defined as solitons except that the metric coefficient is allowed to be
a smooth function rather than a constant. It is shown that any almost soliton is conformal
to another almost soliton having a soliton function which is minus the original one. Uniqueness,
and the case where both the source and target are solitons, are studied. Completeness of the target
metric is also examined in the case where the source is K\"ahler and admits a special
K\"ahler-Ricci potential in the sense of \cite{local,skrp}.
\end{abstract}

\maketitle

\setcounter{Theorem}{0}
\renewcommand{\theTheorem}{\Alph{Theorem}}
\renewcommand{\theequation}{\arabic{section}.\arabic{equation}}
\thispagestyle{empty}
\section{Introduction}
\setcounter{equation}{0}

%\sbe
A gradient Ricci soliton on a manifold $M$ is a Riemannian metric $g$ satisfying
\be\lb{sol}\ri+\nabla df=\cc g,\end{equation} where $\ri$ denotes the Ricci curvature of $g$, $\nabla df$
stands for the $g$-Hessian of a smooth function $f$ and $\cc$ is a constant. We call $f$ the soliton
function, $\cc$ the metric coefficient and $(g,f)$ a soliton pair. If $f$ is constant the soliton is deemed trivial.

Ricci solitons have been intensively studied in recent years (cf. \cite{elm,ca1,pw}), and
their importance derives in part from the role they play in the study of the Ricci flow. More recently,
the more general notion of a gradient Ricci {\em almost} soliton was introduced by Pigola, Rigoli, Rimoldi
and Setti \cite{PRRS}, and further studied in \cite{barrib} and \cite{bbr}. We will often employ the term
``almost soliton" for brevity. Almost solitons and almost
soliton pairs are also defined via \Ref{sol}, but with the metric coefficient $\cc$ an arbitrary smooth
function. In this case we call $(g,f)$ an almost soliton pair, yet $f$ is still called the soliton function.
A more general notion considered in \cite{confsol} is that of a pair $(g,f)$ as above satisfying a Ricci-Hessian
equation \be\lb{ric-hes}\ri+\al\nab df=\cc\, g, \end{equation} in which the coefficients $\al$, $\cc$ are smooth
functions.

The classical problem of determining whether an Einstein manifold can be mapped conformally onto another
Einstein manifold goes back to Brinkmann \cite{brnk}, and was addressed by many authors (see \cite{kuhn}).
In work carried out very recently, the corresponding question for other metric types has been taken up
by Jauregui and Wylie in \cite{GQE}. In the context of generalized quasi-Einstein metrics, a category which includes almost solitons,
it was found that under certain assumptions a strict classification holds, with metrics having such a ``dual" metric
being certain warped products with a one dimensional base. An important assumption that was made in this classification
is that the conformal diffeomorphism preserves, in a certain sense, the generalized quasi-Einstein structure. Remarkably,
we show in this work that relaxing this condition just a little completely annuls the above-mentioned strictness.
Namely, in the context of almost solitons and a single conformal class, we find a ``dual" almost soliton
conformal to {\em any} given one (and, in particular, to any given Ricci soliton). In particular, this gives
many new examples of Ricci almost solitons.

More precisely, we have,
\begin{Theorem}\lb{thm}
In the conformal class of an almost soliton $g$ on a manifold $M$ of dimension $n>2$, there exists another almost soliton $\wht{g}$,
which is nonhomothetic to $g$ if $g$ is nontrivial. The metric $\wht{g}$ is unique among almost solitons in the
conformal class for which both their conformal factor to $g$ and their soliton function are smooth functions
of the soliton function of $g$. If $g$ is complete and nontrivial, $g$ and $\wht{g}$ cannot both be gradient Ricci solitons.
\end{Theorem}
%The soliton function of $\wht{g}$ turns out to be minus that of $g$.
In terms of almost soliton pairs,
the dual is given by \be\lb{inv}\text{$(g,f)\ra (\wht{g},\wht{f}):=(e^{-\frac {4f}{n-2}}g,-f)$.}\end{equation}
This should be compared with \be\lb{old-dual} (g,f)\ra (f^{-2}g,f^{-1}).\end{equation}
Both are involutions on the space of pairs, suitably defined. The latter was important in
\cite{confsol} and has an invariant subset consisting of pairs satisfying a Ricci-Hessian equation,
while \Ref{inv} possesses, by Theorem \ref{thm}, a smaller invariant set consisting of gradient Ricci almost soliton pairs.
The relation between the two involutions can be seen by analogy.
The equation describing the Ricci tensor of a metric conformal to a gradient Ricci almost soliton
contains the Hessians of both the conformal factor and the soliton function of the latter metric.
In this formula these Hessians have the following interchangeable role.
Each of the conformal changes in \Ref{old-dual} and \Ref{inv} implies a particular functional relation between the conformal
factor and the soliton function,  which reduces the equation to a simpler one, namely a Ricci-Hessian equation
in one case (with the Hessian of the conformal factor appearing in one of the terms), and an almost soliton
equation in the other case  (with one of the terms being the Hessian of the soliton function).

It is interesting to note that the square root of the conformal factor in \Ref{inv}, namely $e^{-\frac {2f}{n-2}}$, has been
employed in a number of works as an important conformal factor in its own right (aside from \cite {GQE}, see \cite{Weyl-GQE}, \cite{Zhang},
and also in the very recent \cite{PhoSoStu}). Its effect on the first equation mentioned in the previous paragraph,
is a simplification resulting in an equation involving,
instead of a Hessian, the tensor product of a one form with itself, the one form being the differential of the soliton function.
%Another point of reference is that Proposition \ref{c-hat-c} follow closely a pattern set out in \cite{GQE}.
%One merit of the `duality' examined in this work,
%as compared to Corollary 4.10 in \cite{GQE}, is its applicability beyond the realm of
%warped products with a one dimensional base.

The study of both involutions can be applied to the case where $g$ is K\"ahler. However, we will review the fact
that any gradient K\"ahler-Ricci almost soliton is, in fact, a gradient K\"ahler-Ricci soliton. In the case of
a K\"ahler metric conformal to a gradient Ricci almost soliton, a pertinent role is played by a potential for a
Killing vector field. For the involution \Ref{old-dual}, this Killing potential is just the square root $\t$ of
the conformal factor, while for involution \Ref{inv} it is the soliton function $f$. Both of these are the extreme
special cases in the moduli space considered in \cite{and-pol}, a work written in Polish. There, in the context of
K\"ahler metrics conformal to gradient Ricci solitons, both $\t$ and $f$ are regarded as essentially arbitrary functions of
some Killing potential.

%It will be shown below that complete almost solitons $g$ and $\wht{g}$ corresponding to each other under
%\Ref{inv} cannot both be actual gradient Ricci solitons, unless they are both trivial.
One consequence of Theorem \ref{thm} is that a complete gradient K\"ahler-Ricci soliton cannot have an image $\wht{g}$ under the
involution \Ref{inv}, which is K\"ahler with respect to some complex structure. This is in sharp contrast with the
behaviour of the  involution \Ref{old-dual}, as was shown in \cite{confsol}.

As mentioned above, the duality presented here furnishes many new examples of gradient Ricci almost solitons, for example the duals to
known gradient Ricci solitons. After proving Theorem \ref{thm}
via Propositions \ref{duality}-\ref{c-hat-c}, we study in the final section the completeness question for the dual $\wht{g}$, in
the case where $g$ is a gradient K\"ahler-Ricci soliton which is at the same time a metric with a special K\"ahler-Ricci potential,
in the sense of \cite{local}. These solitons belong to a type studied by Koiso \cite{ko}, and also Cao \cite{ca}. In some cases
$\wht{g}$ is complete.

The author thanks the referee for various suggestions for improving the style of this work.

\section{Duality and \sols}\lb{rds}
\setcounter{equation}{0}
%\sbe

\subsection{Conformal changes}
Let $(M,g)$ be a Riemannian manifold of dimension $n$, and
$\t:M\to {\mathbb{R}}\,$ a $\,C^\infty$ function. We write metrics
conformally related to $g$ in the form $\wht{g}=g/\t^2$, and let $\nab$, $\wht{\nab}$
denote the corresponding Levi-Civita connections as well as the metric gradient operators.
The Koszul formula,
\begin{multline*}
2\wht{g}(\wht{\nab}_wv,u)=d_w[\wht{g}(v,u)]+d_v[\wht{g}(w,u)]-d_u[\wht{g}(w,v)]\\
+\wht{g}(v,[u,w])+\wht{g}(u,[w,v])-\wht{g}(w,[v,u])
\end{multline*}
for smooth vector fields $u$, $v$, $w$, with $d_u$ etc. denoting directional derivatives,
yields the following expression for the $\wht{g}$-Hessian of a
function $f:M\ra \mathbb{R}$
%The metric $\wht{g}$ will always be considered on its domain of definition, i.e. the set
%$M\ssminus\t^{-1}(0)$. With respect to $\wht{g}$, the Hessian of a given $C^2$ function
%$f$ on $M\ssminus\,\t^{-1}(0)$ is given by
\begin{equation}\lb{hess}
\begin{array}{l}
\wht\nabla df\,=\,\nabla df\,+\,\t^{-1}[2\,d\t\odot df
-\,g(\nab\t,\nab f)g],
\end{array}
\end{equation}
where $d\t\odot df=(d\t\otimes df+df\otimes d\t)/2$ with $\otimes$ denoting the tensor product.
We will be concerned primarily with the case where $df\we d\t=0$, i.e., at points where
$df\neq 0$, $\t$ is given locally as a composition $\t=H\circ f$ for some smooth
function $H:\mathbb{R}\ra\mathbb{R}$. In this case,
\Ref{hess} becomes \be\lb{hes-f}\wht\nabla df\,=\,\nabla df\,+\,2\t^{-1}\t'\,df\otimes df
-\,\t^{-1}\t'|\nab f|^2g,\end{equation}
%f'\,\nabla d\t\,
%+\left(f''+2\t^{-1}f'\right)\,d\t\otimes d\t-f'\,\t^{-1}\,Q\,g$,
with
$'$ denoting differentiation with respect to $f$ and $|\cdot|$ is the $g$-norm.

%For the particular choice of $f=\t^{-1}$, this expression simplifies:
%\be\lb{observe}
%\wht\nabla d\t^{-1}\,=-\t^{-2}(\,\nabla d\t\,-\,\t^{-1}\,Q\,g),\quad\mathrm{\ if\ }
%\wht{g}=g/\t^2.
%\end{equation}

Also recall the conformal change expression relating the
Ricci tensors of $g$ and $\wht{g}$, with $\Lap$ denoting the Laplace operator:
\be\lb{Ricci}
\wht{\ri}\,=\,\ri\,+\,(n-2)\,\t^{-1}\nabla d\t\,+\,
\left[\t^{-1}\Delta\t\,-\,(n-1)\,\t^{-2}|\nab\t|^2\right]g.\\
\end{equation}
In the case where $\t$ depends locally on $f$, we write the second term on the right in terms of $f$, obtaining
\be\lb{Ricci2}
\wht{\ri}\,=\,\ri\,+\,(n-2)\,\t^{-1}\t'\nabla df\,+(n-2)\,\t^{-1}\t''df\otimes df\,+\,
\left[\t^{-1}\Delta\t\,-\,(n-1)\,\t^{-2}|\nab\t|^2\right]g.\\
\end{equation}

\subsection{Duality for almost solitons}\lb{dual}
With $M$, $g$, $f$ and other notations as above, we consider the case where the pair $(g,f)$
forms an almost soliton pair.

%We record this equation more explicitly as
%\be\lb{ric-hes}\al\nab d\t +\ri=\cc\, g,\quad \mathrm{with\ } \t
%\mathrm{\ nonconstant},
%\end{equation}
%where $\nab d\t$  and $\ri$ are as above, and $\al$, $\cc$ are $C^\infty$
%coefficient functions.
%What we will call duality may be regarded informally
%as an involution on the space of pairs satisfying \Ref{ric-hes}:
\begin{Proposition}\lb{duality}
Let $M$ be a manifold of dimension $n>2$, and suppose $(g,f)$ denotes
an almost soliton pair satisfying \Ref{sol}. Then $(\wht{g},\wht{f}):=(g/\t^2,-f)$,
with $\t = {\displaystyle e^{\frac {2f}{n-2}}}$, is also an almost
soliton pair. The latter pair satisfies $\wht{\ri}+\wht{\nab} d\wht{f} =\wht{\cc}\, \wht{g}$, with
\be\lb{dual-coef}\quad\wht{\cc}=\t^2(\cc+\bet+\delta),
\end{equation}
where $\bet=\t^{-1}\Delta\t\,-\,(n-1)\,\t^{-2}|\nab\t|^2$ and $\delta=\t^{-1}\t'|\nab f|^2$
denote the coefficients of $g$ in \Ref{Ricci} and \Ref{hes-f}, respectively.
\end{Proposition}
At times $\wht{g}$ will be referred to as the almost soliton dual to $g$, and \Ref{inv} will then be called
a duality.
\begin{proof}
Noting that $\wht{f}=-f$, one has,
by \Ref{Ricci2} and \Ref{hes-f},
\begin{eqnarray*}
\wht{\ri}+\,\wht{\nabla} d\wht{f}&=&
\ri\,
+\,(n-2)\,\t^{-1}\t'\nabla df\,+(n-2)\,\t^{-1}\t''df\otimes df\,+\bet g\\
&+&\,\left(-\nabla df\,-\,2\t^{-1}\t'\,df\otimes df
+\,\delta g\right)\\
&=&\,\ri\,+((n-2)\t^{-1}\t'-1 )\,\nabla df
+\left((n-2)\,\t^{-1}\t''-\,2\t^{-1}\t'\right) df\otimes df\\
&+&(\bet+\delta)\t^2\wht{g}\\
&=&\ri\,+(2-1)\,\nabla df+\t^{-1}\left(\fr {4}{n-2}-\fr {4}{n-2}\right) df\otimes df
+(\bet+\delta)\t^2\wht{g}\\
&=&
\ri\,+\,\nabla df+(\bet+\delta)\t^2\wht{g}=(\cc+\bet+\delta)\t^2\wht{g},
\end{eqnarray*}
where in the penultimate equality we have used the expression for $\t$ as a function of $f$.
\end{proof}
\subsection{Uniqueness}
We show here that among all choices of a conformal factor $\t$ and a soliton function $\wht{f}$
in which both are functions of the initial soliton function $f$, (essentially) only the combination of these
two given in \Ref{inv} gives rise to a dual almost soliton.
\begin{Proposition}\lb{unq}
On a manifold $M$ of dimension $n>2$, let $g$ be a gradient Ricci almost soliton with soliton function $f$.
Suppose $\t (f)$, $k(f)$ are smooth functions of a real variable, regarded as functions on $M$ via composition
with $f$, and $\t (f)$ is nonconstant. Assume $(\t (f))^{-2}g$ is a gradient Ricci almost soliton with soliton
function $\wht{f}=k(f)$. Then $\t (f)=e^{\fr {2f}{n-2}}$ and $k(f)=-f$
up to an additive or, respectively, a multiplicative constant.
\end{Proposition}
\begin{proof}
A computation similar to the one in Proposition \ref{duality} yields that for a dual almost soliton
we have, for the coefficients of $\nab df$ and $df\otimes df$, the following two equations, respectively:
$$(n-2)\t^{-1}\t'+k'=1,\quad  (n-2)\t^{-1}\t''+k''+2\t^{-1}\t'k'=0.$$
Subtracting the derivative of the first equation from the second equation yields, after
dividing by the nonzero coefficient $\t'$ and rearranging, the equation \be\lb{temp}(\log\t)'=-\fr 2{n-2}k'.\end{equation}
Substituting this in the first equation yields $k'=-1$, so that $k(f)=-f$ up to an additive constant,
and then from \Ref{temp} we have $\t(f)=e^{\fr {2f}{n-2}}$, up to a multiplicative constant.
\end{proof}

\subsection{The mapping problem for solitons}

%As mentioned in the introduction, a significant special case of Proposition \ref{duality} is the assertion
%that every gradient Ricci soliton is conformal to a gradient Ricci almost soliton whose soliton
%function is minus the original one.
We now consider whether the image of a gradient Ricci soliton under the map \Ref{inv} can itself be a soliton.
It will turn out that this can happen only rarely. First, note that a gradient Ricci soliton satisfies
\be\lb{sol-const}\Lap f-|\nab f|^2=-2\cc f+k\end{equation}
for a constant $k$, with $f$ and $\cc$ as in \Ref{sol}.
This relation given by Hamilton (cf. \cite{R0}) is deduced by combining the trace of the soliton equation with a
computation involving the Ricci identity.
%This happens when relation \Ref{dual-coef} holds with both $c$ and $\wht{c}$ constant.

Recall that a gradient Ricci soliton is called steady if its metric coefficient function vanishes.
\begin{Proposition}\lb{c-hat-c}
Let $M$ be a manifold of dimension $n>2$. If, on $M$,
the image of a gradient Ricci soliton pair $(g,f)$ satisfying \Ref{sol} under the map \Ref{inv} is also a gradient Ricci soliton
pair, then either $f$ is constant or both metrics are incomplete steady solitons and the constant $k$ of \Ref{sol-const} is zero.
\end{Proposition}
\begin{proof}
The proof resembles Proposition 6.1 in \cite{GQE}. In fact, writing the metric coefficient $\wht{\cc}$ of \Ref{dual-coef} explicitly, we have
$$\wht{\cc}=\t^2\left(\cc+\t^{-1}\Lap\t-(n-1)\t^{-2}|\nab\t|^2+\t^{-1}\t'|\nab f|^2\right),$$
with $\t=e^{\frac {2f}{n-2}}$. Using the latter relation along with $|\nab\t|^2=(\t' )^2|\nab f|^2$ and $\Lap\t=\t'\Lap f+\t''|\nab f|^2$ gives
$$\wht{\cc}=e^{\fr {4f}{n-2}}\left(\cc+\fr 2{n-2}\Lap f+\left(-4\frac {n-1}{(n-2)^2}+\frac 2{n-2}+\frac 4{(n-2)^2}\right)|\nab f|^2\right),$$
and so \be\lb{const}\wht{\cc}=e^{\fr {4f}{n-2}}\left(\cc+\fr 2{n-2}\left(\Lap f-|\nab f|^2\right)\right).\end{equation}

Equations \Ref{const} and \Ref{sol-const} both hold exactly when \be\lb{lap-equal}\frac {n-2}2(\wht{\cc}\,e^{\fr {-4f}{n-2}}-\cc)=-2\cc f+k.\end{equation}
Applying the operator $d$ to this equation, with $\cc$ and $\wht{\cc}$ constant, yields
that solutions only occur for $n\neq 2$ if $f$ is constant or $\cc=\wht{\cc}=0$. In the latter case,
we proceed as in \cite{GQE}, Namely, we have $\Lap f-|\nab f|^2=0$ by \Ref{const}. Thus, because the scalar curvature $R$ of a steady
($\cc=0$) soliton satisfies $R=-\Lap f$, we get $-R-|\nab f|^2=0$.
Now it is known \cite{ch} that for $g$ a complete steady soliton, $R\geq 0$. Consequently, in
that case $|\nab f|^2=0$, so that again $f$ must be constant.
Thus if it is not constant, $g$ must be incomplete, and since $g$ and $\wht{g}$ are symmetric
with respect to the involution \Ref{inv}, the metric $\wht{g}$ is also
incomplete. Finally, from Equation \Ref{lap-equal}, if
$\cc=\wht{\cc}=0$ then $k=0$ as well.
\end{proof}

%Note that for a gradient Ricci soliton $g$ with soliton function $f$, the variation of the function $\wht{c}$ associated to $\wht{g}$ is controlled by
%the variation of $f$. Namely, considering Equation \Ref{const} at the maximum $f_{\mathrm{max}}$ and minimum $f_{\mathrm{min}}$
%of $f$ we have $$c\,e^{{\textstyle \fr {4f_\mathrm{min}}{n-2}}}\leq \wht{c}\leq c\,e^{{\textstyle\fr {4f_\mathrm{max}}{n-2}}}.$$

\section{The K\"ahler case}
\subsection{Generalities}
The K\"ahler condition forces an almost soliton to be a soliton.
\begin{Proposition}\lb{Kahler-almost}
Let $M$ be a manifold of even dimension $n\geq 4$.
A gradient Ricci almost soliton pair $(g,f)$  on $M$, for which $g$ is K\"ahler with respect to some
complex structure $J$, is in fact a gradient K\"ahler-Ricci soliton.
\end{Proposition}
This follows from Lemma 6.1(ii) of \cite{local}. For completeness, we give the proof
in a form that resembles Proposition 3.3 of \cite{confsol}. For this, recall that a Killing potential
$f$ on a K\"ahler manifold $(M,J,g)$ (with $J$ the almost complex structure), is a smooth function for which $u:=J\nab f$
is a Killing vector field, i.e. $L_ug=0$ where $L_u$ is the Lie derivative.
\begin{proof}
From \Ref{sol}, since $g$ and the Ricci curvature $\ri$ are hermitian, so is $\nab df$.
By \cite[Lemma 5.2]{local}, this implies that $f$ is a Killing potential.
By \cite[Lemma 5.5]{local}, this in turn implies that $\nab df(J\cdot,\cdot)=d(\imath_{\nab f}\om)/2$.

Now compose Equation \Ref{sol} with $J$ and apply the differential operator $d$ to the result.
Using the conclusion of the previous paragraph and the fact that the K\"ahler form $\om$ and the Ricci form are closed,
one arrives at $0=d\cc\we\om$. But the operation $\we\om$ is injective on $1$-forms in dimensions four and above.
Hence $d\cc=0$.
\end{proof}
Note that for {\em any} gradient Ricci almost soliton it is known that the metric coefficient function is
locally a function of the soliton function (see Remark 2.5 in \cite{PRRS}).

As mentioned in the introduction, unlike the situation for the involution \Ref{old-dual},
Propositions \ref{duality}, \ref{c-hat-c} and \ref{Kahler-almost} together imply
that in the K\"ahler setting, the property of having a dual complex structure does not hold for the involution \Ref{inv}:
\begin{Corollary}\lb{two-Kahler}
Given a nontrivial complete gradient K\"ahler-Ricci soliton $g$ with soliton function $f$, the metric $e^{\frac {-4f}{n-2}}g$
is not K\"ahler with respect to any complex structure.
\end{Corollary}
In fact, we know that the metric in question is an almost soliton. If it were K\"ahler, it would be a K\"ahler-Ricci soliton. Both
this metric and $g$ would thus be gradient Ricci solitons, but in that case $g$ cannot be both complete and nontrivial.

The propositions above can be applied to the study of the setting analogous to the one considered in \cite{confsol},
namely that of a K\"ahler metric conformal to a gradient Ricci almost soliton.
Note that, as the latter is the target metric in this setting, in the next proposition our notations for the metrics
are reversed from those in the rest of this work.
\begin{Proposition}
Suppose $g$ is a K\"ahler metric on a complex manifold $(M,J)$ of dimension $n\geq 4$,
which is conformal to a gradient Ricci almost soliton $\wht{g}:=g/\t^2$ with
soliton function $f$. If $f$ is a $g$-Killing potential and $\t$ is locally a function of $f$ with $\t(0)=1$ and
$\t'(0)=-\fr 2{n-2}$, then $g$ is a K\"ahler-Ricci soliton.
\end{Proposition}
\begin{proof}
Since $g$ is conformal to an almost soliton, \Ref{hess} and \Ref{Ricci} yield
\begin{eqnarray*}
r &+& (n-2)\t^{-1}\nab d\t+\nab df+2\t^{-1}d\t\odot df\\
&=&[\wht{\cc}\t^{-2}+(n-1)\t^{-2}|\nab\t|^2-\t^{-1}\Lap\t+\t^{-1}g(\nab\t,\nab f)]g,
\end{eqnarray*}
with $\wht{\cc}$ the almost soliton coefficient function of $(\wht{g},f)$. If $d\t\we df=0$, we regard $\t$ as a function
of $f$, and this expression becomes, in similarity with Proposition \ref{duality},
$$\ri\,+((n-2)\t^{-1}\t'+1 )\,\nabla df
+\left((n-2)\,\t^{-1}\t''+\,2\t^{-1}\t'\right) df\otimes df
=\left(\wht{\cc}\t^{-2}-\bet+\delta\right)g,$$
with $\bet$, $\delta$ as in that proposition.
As $f$ is a Killing potential for the K\"ahler metric $g$, the Hessian $\nab df$ is hermitian with
respect to the associated complex structure, and thus
all the terms in the last equation are hermitian except for $df\otimes df$.
Thus for the latter term, the coefficient $(n-2)\,\t^{-1}\t''+\,2\t^{-1}\t'$ must vanish, giving an ODE whose
solution with the given initial conditions is $\t=\exp (-2f/(n-2))$.
We can now apply Proposition \ref{duality} (with $\wht{g}$ and $g$ interchanged), to conclude that
$g$ is a gradient Ricci almost soliton. As $g$ is K\"ahler, Proposition \ref{Kahler-almost}
yields that it is in fact a gradient K\"ahler-Ricci soliton.
\end{proof}

\subsection{Example: Ricci solitons that admit a special K\"ahler-Ricci potential}

Metrics with special K\"ahler-Ricci potentials have been studied in \cite{local,skrp}
and \cite{confsol}. If a gradient Ricci soliton also admits such a potential,
it is of a type first considered by Koiso \cite{ko} , and also Cao \cite{ca}.
We give here an independent viewpoint on such metrics and then consider their dual under
the map induced by \Ref{inv}.

A K\"ahler metric $g$ on a complex manifold $(M,J)$ admits a special K\"ahler-Ricci potential $f$ if
$f$ is a Killing potential and, at each noncritical point of $f$, all nonzero tangent vectors
orthogonal to the complex span of $\nab f$ are eigenvectors of both the Ricci tensor and the Hessian
of $f$. This rather technical definition implies, by \cite[Remark 7.4]{local}, the existence on an open set,
of a Ricci-Hessian equation \Ref{ric-hes}, which we reproduce here:
\be\lb{ric-hes-2}\ri+\al\nab df=\cc\, g.
\end{equation}
The coefficients $\al$, $\cc$ are smooth functions, and in fact locally functions of $f$.
The metric $g$ is called an \sk\ metric and $(g,f)$ is called an \sk\ pair.

We turn next to the local classification of metrics with a special K\"ahler-Ricci potential.
We consider only the non-trivial case, of metrics that are not
local products of K\"ahler metrics. In this case the metrics come in the following families involving
one free function, described using a variant of the Calabi construction.

Let $\pi:(L,\langle\cdot,\cdot\rangle)\ra (N,h)$ be a Hermitian holomorphic line bundle
over a K\"ahler-Einstein manifold of complex dimension $m-1$. Assume that the
curvature of $\langle\cdot,\cdot\rangle$ is a multiple of the K\"ahler form of $h$.
%Note that, if $N$ is compact and $h$ is not Ricci flat, this implies that $L$ is smoothly
%isomorphic to a rational power of the anti-canonical bundle of $N$.
Consider, on $L\ssminus N$ (the total space of $L$ excluding the zero section),
the metric $g$ given by
\be\lb{metric-form}
g|_{\cal{H}}=2|f_c|\,\pi^*h, \quad g|_{\cal{V}}=\fr {Q(f)}{(p\cdot \ell)^2}\,
\mathrm{Re}\,\langle\cdot,\cdot\rangle,
\end{equation}
where\\
-- $\cal{V},\cal{H}$ are the vertical/horizontal distributions of $L$, respectively,
the latter determined via the Chern connection of $\langle\cdot,\cdot\rangle$,\\
-- $\ell$ is the norm induced by $\langle\cdot,\cdot\rangle$,\\
-- $f$ is a function on $L\ssminus N$ obtained as follows: one fixes an open interval $I$ and a positive
$C^\infty$ function $Q(f)$ on $I$,
solves the differential equation $(a/Q)\,df=d(\log \ell)$ to obtain a diffeomorphism
$\ell(f):I\ra (0,\infty)$, and defines $f(\ell)$  as the inverse of this diffeomorphism
(composed on the norm $\ell$),\\
-- $c$ and $p\neq 0$ are constants and $f_c=f-c$.\\

For $m\geq 2$, the pair $(g,f)$ is a nontrivial \sk\ pair.
%If $g$ is nontrivial, the connection on $L$ will not be flat.
%The constant $\kappa$ of Remark \ref{c-const} is the Einstein constant of $h$,
%so that if $g$ is nontrivial, it is in fact standard (for an arbitrary \sk\
%metric, $h$ need not be Einstein if $m=2$). For $g$ standard, or merely nontrivial,
%the \sk\ constant $c$ (see again Remark \ref{c-const}) coincides with $c$ of
%\Ref{metric-form}.
Conversely, for any nontrivial \sk\ metric $(M,J, g,f)$ with $m>2$, any point
that is not a critical point for $f$ has a neighborhood biholomorphically isometric to
an open set in some triple $(L,g,f(\ell))$ as above (this is a special case of
\cite[Theorem 18.1]{local}).

The function $Q$ of a nontrivial \sk\ metric as above is given by $Q(f)=2f_c\,\phi (f)$,
where $\phi (f)$ is a solution to the ordinary differential equation derived from \Ref{ric-hes-2}.
Namely,
\begin{equation}\lb{mek}
(f_c)^2\phi''+\,(f_c)[ m-(f_c)\al]\phi'-\,m\phi\,=\,-\mathrm{sgn}(\phi)\kappa/2
\end{equation}
holds at points which are not critical for $f$ and for which  $\phi'(f)$ is nonzero,
with $\kappa$ the Einstein constant of the metric $h$ \cite[Proposition 4.1]{confsol}.

We now suppose that the Ricci-Hessian equation of an \sk\ metric also defines a gradient
Ricci-soliton, in other words $\al=1$ and $\cc$ is constant in \Ref{ric-hes-2}.
Substituting $\al=1$ in \Ref{mek}, the resulting differential equation has solutions $\phi(f)$, which we
take for positive $\phi$, given by
\be\lb{solsols}
\phi(f)=\frac 1 {f_c^m}\left(A\sum_{k=0}^m\frac {f_c^k}{k!}+Be^{f_c}\right)-\frac \kappa {2m}
\end{equation}
for constants $A$, $B$. One can verify that for these $\phi$ the function $\cc$ is indeed constant,
as it must be by Proposition \ref{Kahler-almost}, from the formula
$$\cc=\al\phi+\left(\al f_c-(m+1)\right)\phi'-f_c\phi'',$$
valid for any such \sk\ metric (see \cite[Section 4.2]{confsol}).

The almost soliton dual under \Ref{inv} to an \sk\ metric of this type is thus given by
\be\lb{metric-dual}
\wht{g}|_{\cal{H}}=2|f_c|e^{\frac {-4f}{n-2}}\,\pi^*h, \quad \wht{g}|_{\cal{V}}=\fr {Q(f)}{(p\cdot \ell)^2}\,e^{\frac {-4f}{n-2}}\,
\mathrm{Re}\,\langle\cdot,\cdot\rangle.
\end{equation}
We now examine conditions under which $\wht{g}$ is complete. Assuming the
manifold $N$ is compact, it is enough to check completeness
for the $\wht{g}$-geodesics normal to, say, the zero section of $L$, an $f$-critical manifold.
Since the conformal factor is a function of $f$, it follows that these geodesics
coincide, as unparametrized curves, with the $g$-geodesics normal to the zero section.

In fact, unparametrized  $g$-geodesics form integral curves of $v:=\nabla f$ (see \cite[Section 8]{skrp}), and
thus $\nabla_vv=\alpha v$ for some function $\alpha$. Now for vector fields $x$, $y$ on $M$, we have the conformal change formula
$\wht{\nabla}_xy=\nabla_xy-(d_x f)y-(d_y f)x+g(x,y)\nabla f$. Thus  $\wht{\nabla}_vv=(\alpha-2d_vf+|\nab f|^2)v:=\beta v$, and therefore the
$\wht{g}$-geodesics also coincide with integral curves of $\nabla f$.

To obtain the arc length formula, we first reparametrize an integral curve
$\tilde{x}(t)$ of $\nab f$ so that $f$ itself is the new parameter,
giving the form $x(f)$. Then the velocity is given by
$x'(f)=\psi'(f)\tilde{x}'(\psi(f))=\psi'(f)\nabla f|_{\tilde{x}(\psi(f))}$, where $t=\psi(f)$ is the reparametrization
map and the prime denotes differentiation. To compute $\psi'(f)$ we apply $g(\cdot,\nab f)$ to this equation, giving
$$1=\frac d{df} f = d_{x'} f = g(x',\nabla f) =\psi'(f)g(\nabla f,\nabla f):=\psi'(f)Q.$$
Thus the velocity $x'$ is given by $\frac 1 Q\nabla f$. (The notation $Q$ corresponds to
$Q(f)$, as in fact the expression $g(\nabla f,\nabla f)$ is locally a function of $f$ (see \cite[Lemmas 7.5 and 11.1]{local}),
namely $Q(f)$ of Equation \Ref{metric-dual}.) Hence the $\wht{g}$-arc length $s$ of the curve $x(f)$
is characterized by $$\frac {ds}{df}=\sqrt{\wht{g}(x'(f),x'(f))}=\sqrt{e^{-\frac{4f}{n-2}}g(x'(f),x'(f))}=e^{-\frac{2f}{n-2}}\sqrt{\frac 1{Q^2} Q}=
e^{-\frac{2f}{n-2}}\frac 1{\sqrt{Q}}.$$
Completeness thus depends on the non-integrability of $\displaystyle{\int_{\inf f}^{\sup f} e^{-\frac{2f}{n-2}}\frac 1{\sqrt{Q}}df}$,
where the infimum and supremum are with respect to the range of values of $f$.

We distinguish a few cases. If the range of $f$ is a finite interval, we focus, say, on the endpoint $a:=\inf f$.
An infinite contribution to the integral near $a$ can clearly only occur if $a$ is a zero of $Q$. Assuming we have such a zero,
if $Q'(a)\ne 0$ then $1/\sqrt{Q}$
is asymptotic to $(f-a)^{-1/2}$ near $a$, so that the integral is finite. However, under these conditions $g$ extends smoothly
to the level set $\{f=a\}$ (see Remark $4.3$ and Lemma $4.4$ of \cite{skrp}), hence $\wht{g}$ extends as well and the
finite length curve $x(f)$ has a limit at $a$, so that, assuming the same behaviour at $\sup f$, we see that $\wht{g}$ is complete
(in fact the manifold is then compact).

Next, if $Q(a)=Q'(a)=0$, then $1/\sqrt{Q}$ is asymptotic to $(f-a)^p$ with $p\leq -1$ so that the integral diverges and $x(t)$
has infinite length to the critical set. If this occurs also at $\sup f$ then $\wht{g}$ is complete.

There remains the case where, say, $\sup f=\infty$. Making the substitution $u=1/f_c$ gives, using the explicit
expression \Ref{solsols} for $\phi(f)$ in $Q(f)=2f_c\,\phi(f)$:
\begin{multline}
\int_{1}^{\infty} e^{-\frac{2f}{n-2}}\frac 1{\sqrt{Q}}df =
- \int_0^1 e^{-\frac 1{m-1}u^{-1}}e^\frac c{m-1}
u^{-\fr {m-1}2}\left(A\textstyle{\sum}_{k=0}^m\frac {u^{-k}}{k!}+Be^{(u^{-1})}\right)^{-\fr 12}u^{-2}\,du.
\end{multline}
Assuming $B\ne 0$ (the analysis of the case $B=0$ is similar), neglecting the polynomial in $u^{-1}$, the integrand of the last expression is asymptotic to
$$\exp\left(\left(-\frac 1{m-1}-\frac 12\right)u^{-1}\right) u^{-\left(\frac {m-1}2+2\right)}$$
as $u\rightarrow 0^+$.
Over, say, the interval $[0,1]$, the integral of this quantity converges for $m>1$, so that $\wht{g}$ is incomplete.


\begin{thebibliography}{99}

%\bibitem{ham1} Apostolov, V., Calderbank, D.M.J., Gauduchon P.: Hamiltonian 2-forms
%in K\"ahler geometry, I. General theory.  J. Differential Geom.  73  (2006), 359--412

%\bibitem{ham2} Apostolov, V., Calderbank, D.M.J., Gauduchon P., T{\o}nnesen-Friedman
%C. W.:  Hamiltonian 2-forms in K\"ahler geometry, II.
%Global classification. J. Differential Geom. 68 (2004), 277--345

%\bibitem{ham4} Apostolov, V., Calderbank, D.M.J., Gauduchon P., T{\o}nnesen-Friedman
%C. W.: Hamiltonian 2-forms in K\"ahler geometry, IV. Weakly Bochner-flat K\"ahler
%manifolds. In: arXiv:math.DG/0511119, to appear in Comm. Anal. Geom.

\bibitem{bbr} Barros, A., Batista, R., Ribeiro, E., Jr:
Compact almost Ricci solitons with constant scalar curvature are gradient.
arXiv:1209.2720.

\bibitem{barrib} Barros, A., Ribeiro, E., Jr.:
Some characterizations for compact almost Ricci solitons.
Proc. Amer. Math. Soc. 140 (2012), 1033–-1040.

\bibitem{brnk} Brinkmann, H. W.: Einstein spaces which are mapped conformally on each other.
Math. Ann. 94 (1925),  119–-145.

\bibitem{ca} Cao, H.-D.: Existence of gradient K\"ahler-Ricci solitons.
In: Elliptic and parabolic methods in geometry, Minneapolis MN 1994, 1--16
A. K. Peters, Wellesley MA (1996).

\bibitem{ca1} Cao, H.-D.: Recent progress on Ricci solitons.
Recent advances in geometric analysis, 1--38, Adv. Lect. Math. (ALM), 11, Int. Press, Somerville, MA, 2010.

\bibitem{Weyl-GQE} Catino, G.: Generalized quasi-Einstein manifolds with harmonic Weyl tensor.
Math. Z. 271 (2012), 751--756.

%\bibitem{clw} Chen, X., LeBrun, C., Brian Weber B.: On Conformally
%K\"ahler, Einstein Manifolds.  In: J. Amer. Math. Soc., posted on Jan. 28, 2008,
%PII: S 0894-0347(08)00594-8 (to appear in print)

\bibitem{ch} Chen, B.-L., Strong uniqueness of the Ricci flow. J.
Differential Geom. 82 (2009), 363--382.

\bibitem{R0} Chow, B., Chu, S.-C., Glickenstein, D., Guenther, C., Isenberg, J.,
Ivey, T., Knopf, D., Lu, P., Luo, F., Ni, L.: The Ricci flow: techniques
and applications. Part I, Mathematical Surveys and Monographs, vol. 135, American
Mathematical Society, Providence, RI, 2007.

\bibitem{and-pol} Derdzinski, A.: Solitony Ricciego, Wiadomosci Matematyczne,
48 (2012),  1--32.

\bibitem{local} Derdzinski, A., Maschler, G.: Local classification
of conformally-Einstein K\"ahler metrics in higher dimensions. Proc. London
Math. Soc. 87 (2003), 779--819.

\bibitem{skrp} Derdzinski, A., Maschler, G.: Special K\"ahler-Ricci
potentials on compact K\"ahler manifolds. J. reine angew. Math. 593 (2006), 73--116.

%\bibitem{global} Derdzinski, A., Maschler, G.: A moduli curve for compact
%conformally-Einstein K\"ahler manifolds. Compositio Math. 141 (2005), 1029--1080

\bibitem{elm} Eminenti, M., La Nave, G.; Mantegazza, C.:
Ricci solitons: the equation point of view. Manuscripta Math. 127 (2008), 345--367.

\bibitem{hamilton-tr} Hamilton, R.S.: The Ricci flow on surfaces.
In: Mathematics and general relativity, Santa Cruz CA, 1986.
Contemp. Math., vol. 71, pp. 237--262. AMS, Providence RI (1988).

\bibitem{GQE} Jauregui, J., Wylie, W.: Conformal diffeomorphisms of gradient
Ricci solitons and generalized quasi-Einstein manifolds. arXiv:1209.1118v1.

\bibitem{ko} Koiso N.: On rotationally symmetric Hamilton's equation for
K\"ahler-Einstein metrics. In: Recent topics in differential and analytic geometry,
Adv. Stud. Pure Math., vol. 18-I, pp. 327--337. Academic Press, Boston MA (1990).

\bibitem{kuhn} K\"uhnel, Wolfgang: Conformal transformations between Einstein spaces.
In: Conformal geometry (Bonn, 1985/1986), 105–146, Aspects Math., E12, Vieweg,
Braunschweig, 1988.

\bibitem{confsol} Maschler, G: Special K\"ahler-Ricci potentials and Ricci solitons.
Ann. Global Anal. Geom. 34 (2008), 367--380.

%\bibitem{p} Perelman, G.: The entropy formula for the
%Ricci flow and its geometric applications. In: arXiv:math.DG/0211159

\bibitem{pw} Petersen, P., Wylie, W.
On the classification of gradient Ricci solitons. Geom. Topol. 14 (2010), 2277--2300.

\bibitem{PhoSoStu} Phong D. H., Song J., Sturm, J.:
Degeneration of K\"ahler-Ricci solitons on Fano manifolds.
arXiv:1211.5849.

\bibitem{PRRS} Pigola, S., Rigoli, M., Rimoldi, M., Setti, A.
Ricci almost solitons.
Ann. Sc. Norm. Super. Pisa Cl. Sci. (5) 10 (2011), 757--799.

%\bibitem{wz} Wang, X.-J., Zhu, X.: K\"ahler-Ricci solitons on toric
%manifolds with positive first Chern class. Adv. Math.  188 (2004), 87--103

\bibitem{Zhang} Zhang, Z.: Degeneration of shrinking Ricci solitons.
Int. Math. Res. Not. IMRN (2010), No. 21, 4137--4158.




\end{thebibliography}
\end{document}